\documentclass[preprint,12pt]{elsarticle}

\usepackage[T1]{fontenc}
\usepackage[utf8]{inputenc}
\usepackage{amsmath,amssymb,amsthm,mathtools}
\usepackage{aliascnt}
\usepackage{microtype}
\usepackage[hidelinks]{hyperref}
\usepackage[nameinlink,noabbrev]{cleveref}

\allowdisplaybreaks[3]
\biboptions{numbers,sort&compress}

\newtheorem{theorem}{Theorem}[section]

\newaliascnt{lemma}{theorem}
\newtheorem{lemma}[lemma]{Lemma}
\aliascntresetthe{lemma}

\newaliascnt{proposition}{theorem}
\newtheorem{proposition}[proposition]{Proposition}
\aliascntresetthe{proposition}

\newaliascnt{corollary}{theorem}

\aliascntresetthe{corollary}

\newtheorem*{theorem*}{Theorem}

\theoremstyle{definition}
\newaliascnt{definition}{theorem}
\newtheorem{definition}[definition]{Definition}
\aliascntresetthe{definition}

\theoremstyle{remark}
\newaliascnt{remark}{theorem}

\aliascntresetthe{remark}

\crefname{theorem}{theorem}{theorems}
\Crefname{theorem}{Theorem}{Theorems}
\crefname{lemma}{lemma}{lemmas}
\Crefname{lemma}{Lemma}{Lemmas}
\crefname{proposition}{proposition}{propositions}
\Crefname{proposition}{Proposition}{Propositions}
\crefname{corollary}{corollary}{corollaries}
\Crefname{corollary}{Corollary}{Corollaries}
\crefname{definition}{definition}{definitions}
\Crefname{definition}{Definition}{Definitions}
\crefname{remark}{remark}{remarks}
\Crefname{remark}{Remark}{Remarks}

\newcommand{\N}{\mathbb{N}}
\newcommand{\Z}{\mathbb{Z}}
\newcommand{\Pcal}{\mathcal{P}}
\newcommand{\Fcal}{\mathcal{F}}
\newcommand{\lad}{\mathrm{ad}}
\newcommand{\1}{\mathbf{1}}
\newcommand{\essinf}{\operatorname*{ess\,inf}}
\newcommand{\esssup}{\operatorname*{ess\,sup}}
\newcommand{\norm}[1]{\left\lVert #1\right\rVert}

\newcommand{\lesssimc}{\mathrel{\lesssim}}
\newcommand{\approxsim}{\mathrel{\approx}}

\journal{Journal of Mathematical Analysis and Applications}

\begin{document}

\begin{frontmatter}

\title{Real interpolation for adapted sequence spaces with variable exponents}

\author[uncc]{Asad Ullah\corref{cor1}}
\cortext[cor1]{Corresponding author.}
\ead{aullah@uncc.edu}

\address[uncc]{Department of Mathematics and Statistics,
University of North Carolina at Charlotte,
9201 University City Blvd., Charlotte, North Carolina 28223, USA}

\begin{abstract}
We study real interpolation for adapted sequence spaces with variable
exponents. Let
\((\Omega,\Fcal,\mathbb{P};(\Fcal_n)_{n\geq1})\) be a filtered complete
probability space, let \(p(\cdot)\in\Pcal(\Omega)\), and let
\(0<q\leq\infty\) and \(0<\theta<1\). We prove that
\[
\bigl(L^{\lad}_{p(\cdot)},L^{\lad}_{\infty}\bigr)_{\theta,q}
=
L^{\lad}_{\widetilde p(\cdot),q},
\qquad
\frac{1}{\widetilde p(\cdot)}
=
\frac{1-\theta}{p(\cdot)},
\]
with equivalent quasi-norms. Here
\(L^{\lad}_{\widetilde p(\cdot),q}\) consists of adapted sequences
\(f=(f_n)_{n\geq1}\) whose square function
\(\sigma(f)=(\sum_{n=1}^{\infty}|f_n|^2)^{1/2}\) belongs to the
variable Lorentz space \(L_{\widetilde p(\cdot),q}\). The proof uses a
decomposition that preserves adaptedness and provides an upper estimate
for the corresponding \(K\)-functional. No continuity condition on the
variable exponent and no measurability relation between \(p(\cdot)\)
and the filtration are required.
\end{abstract}


\end{frontmatter}

\section{Introduction}\label{sec:introduction}

Let
$
p(\cdot):\Omega\longrightarrow(0,\infty)
$
be a measurable function defined on a complete probability
space $(\Omega,\Fcal,\mathbb{P})$. The variable Lebesgue space
$L_{p(\cdot)}(\Omega)$ consists of measurable functions whose
integrability is governed pointwise by the exponent $p(\cdot)$.
When $p(\cdot)\equiv p$ is constant, the space
$L_{p(\cdot)}(\Omega)$ coincides with the classical Lebesgue
space $L_p(\Omega)$. Variable exponent spaces have been studied
extensively because of their applications to harmonic analysis,
partial differential equations, and variational problems; see,
for example,
\cite{CruzUribeFiorenza,DieningBook,KovacikRakosnik,NakaiSawano}.
Real interpolation for variable exponent spaces was developed
by Kempka and Vyb\'iral \cite{KempkaVybiral}. In particular,
they identified the interpolation space between a variable
Lebesgue space and $L_\infty$ as a variable Lorentz space. In the martingale setting, Jiao, Weisz, Zhou, and Wu
\cite{JiaoEtAlInterpolation} developed variable martingale Hardy
and Lorentz--Hardy spaces, established atomic decompositions and
martingale inequalities, and obtained applications to Fourier
analysis.

The use of square-summable adapted sequences is classical in
martingale theory. The classical inequality of L\'epingle
\cite{Lepingle} compares the $\ell_2$-norm of the predictable
projections of an adapted sequence with the $\ell_2$-norm of the
original sequence in the classical $L_p$ setting. More recently,
Randrianantoanina \cite{Randrianantoanina} considered adapted
sequences in the general K\"othe--Bochner space
$E(\Omega;\ell_2)$, proved a Davis-type decomposition for such
sequences, and established a L\'epingle inequality under a dual
form of Doob's maximal inequality. Taking
$E=L_{p(\cdot)}$ gives the variable-exponent ambient space that
underlies the present work. To the best of our knowledge, the
specific adapted variable Lebesgue and adapted variable Lorentz
sequence spaces $L^{\lad}_{p(\cdot)}$ and
$L^{\lad}_{p(\cdot),q}$ are isolated and named for the first time
in this paper. Consequently, to the best of our knowledge, no
earlier structural or interpolation results are known
specifically for these spaces. They are motivated by the natural
question of whether the scalar interpolation identity for
$(L_{p(\cdot)},L_\infty)$ lifts to the closed subspace determined
only by adaptedness, without imposing martingale cancellation.
\Cref{prop:complete} therefore establishes the first basic
structural property of $L^{\lad}_{p(\cdot)}$: it is an isometric
closed subspace of $L_{p(\cdot)}(\Omega;\ell_2)$ and hence a
quasi-Banach space, while $L^{\lad}_{p(\cdot),q}$ is obtained by
pulling back the variable Lorentz quasi-norm through the
square-function map $f\mapsto\sigma(f)$.

For a given martingale \(f=(f_n)_{n\geq0}\), the corresponding
square function is denoted by
\[
s(f)
:=
\left(
\sum_{n\geq0}
\mathbb{E}_{n-1}|d_nf|^2
\right)^{1/2}.
\]

The classical martingale Hardy space \(H_p^s\) is defined by
\[
H_p^s
:=
\left\{
f=(f_n):
\norm{f}_{H_p^s}
:=
\norm{s(f)}_p
<\infty
\right\}.
\]

In the classical martingale setting, Weisz~\cite{Weisz}
established, using the Hardy inequality and Theorem 5.8
of~\cite{Weisz}, that
\[
(H_p^s,H_\infty^s)_{\theta,q}
=
H_{p/(1-\theta),q}^s,
\qquad
0<q\leq\infty,
\quad
0<\theta<1.
\]

The variable-exponent martingale theory developed in
\cite{JiaoEtAlInterpolation} provides a natural comparison point,
but the present work concerns arbitrary adapted sequences rather
than martingale difference sequences. This distinction removes the
cancellation structure used in martingale Hardy spaces and leads to
a separate filtration-compatible truncation problem.

Recall that a sequence
\(f=(f_n)_{n\geq1}\subset L_0(\Omega)\) is said to be adapted
if each \(f_n\) is \(\Fcal_n\)-measurable. Let
$(\Fcal_n)_{n\geq1}$ be an increasing sequence of
sub-$\sigma$-algebras of $\Fcal$, and let
$f=(f_n)_{n\geq1}$ be adapted to this filtration. Set
\[
\sigma_n(f)
=
\left(
\sum_{k=1}^{n}|f_k|^2
\right)^{1/2},
\qquad
\sigma(f)
=
\left(
\sum_{k=1}^{\infty}|f_k|^2
\right)^{1/2}.
\]

We then define the space of adapted sequences by
\[
L_{p(\cdot)}^{\mathrm{ad}}(\Omega)
:=
\left\{
g=(g_n):
g\text{ is adapted},
\ \norm{\sigma(g)}_{p(\cdot)}<\infty
\right\}.
\]

We now consider the interpolation of adapted sequence spaces.
Analogously to the Hardy-space setting, we aim to prove the
following theorem.

\begin{theorem}\label{thm:main}
Let $p(\cdot)\in\Pcal(\Omega)$,
$0<q\leq\infty$, $0<\theta<1$, and
\[
\frac{1}{\widetilde p(\cdot)}
:=
\frac{1-\theta}{p(\cdot)}.
\]
Then
\[
\bigl(
L^{\lad}_{p(\cdot)},
L^{\lad}_{\infty}
\bigr)_{\theta,q}
=
L^{\lad}_{\widetilde p(\cdot),q}
\]
with equivalent quasi-norms.
\end{theorem}

Here \(L^{\lad}_{\widetilde p(\cdot),q}\) denotes the adapted
variable Lorentz sequence space from
\Cref{def:adapted-spaces}; explicitly, it consists of all
adapted sequences \(f=(f_n)_{n\geq1}\) for which
\(\sigma(f)\in L_{\widetilde p(\cdot),q}\), and
\[
\norm{f}_{L^{\lad}_{\widetilde p(\cdot),q}}
=
\norm{\sigma(f)}_{L_{\widetilde p(\cdot),q}}.
\]

Although the statement is analogous to the Hardy-space result,
the passage from martingales to arbitrary adapted sequences
creates a different obstruction. In the conditional Hardy
setting, the decomposition is performed on martingale
differences by stopping-time sets associated with the
conditional square function, so the martingale-difference
structure is retained. Here there is no cancellation condition,
and the direct scalar truncation by
\(\{\sigma(f)\leq\mu\}\) cannot be applied to the \(n\)th
coordinate because this terminal event need not be
\(\Fcal_n\)-measurable. Moreover, interpolation of the ambient
K\"othe--Bochner spaces does not automatically pass to the
adapted subspace: under our hypotheses, we cannot invoke a
bounded projection onto adapted sequences, and no measurability
relation between \(p(\cdot)\) and the filtration is assumed.

We overcome these points by truncating the \(n\)th coordinate on
the genuinely adapted set
\[
Q_n
=
\{\sigma_n(f)\leq\mu\}.
\]
The sets \((Q_n)\) are decreasing, and a telescoping argument is
needed to show that the truncated sequence has uniformly
bounded square function. The complementary sequence is then
controlled by
\[
\sigma(f)\1_{\{\sigma(f)>\mu\}}.
\]
This adapted truncation is the new ingredient that replaces the
stopping-time and atomic machinery used in the variable Hardy
space setting.

The paper is organized as follows. \Cref{sec:preliminaries}
contains the necessary preliminaries concerning variable
Lebesgue spaces, variable Lorentz spaces, adapted sequence
spaces, and the real interpolation method. In
\Cref{sec:proof-main}, we prove the \(K\)-functional estimate and
then establish \Cref{thm:main}.

Throughout this paper, $\Z$ and $\N$ denote the sets of
integers and positive integers, respectively. The letter $C$
denotes a positive constant that may change from line to line.
We write
$
A\lesssimc B
$
if $A\leq CB$, and
$
A\approxsim B
$
if both $A\lesssimc B$ and $B\lesssimc A$.

\section{Preliminaries}\label{sec:preliminaries}

In this section, we introduce the notation and basic results
used in the proof of the main theorem.

\subsection{Variable Lebesgue spaces and Lorentz spaces}
\label{subsec:variable-spaces}

Let $(\Omega,\Fcal,\mathbb{P})$ be a complete probability
space. A measurable function
\[
p(\cdot):\Omega\longrightarrow(0,\infty)
\]
is called a variable exponent.

For a measurable set $A\subseteq\Omega$, define
\[
p_-(A)
:=
\essinf_{\omega\in A}p(\omega),
\qquad
p_+(A)
:=
\esssup_{\omega\in A}p(\omega),
\]
and write
\[
p_-:=p_-(\Omega),
\qquad
p_+:=p_+(\Omega).
\]

We denote by $\Pcal(\Omega)$ the collection of all variable
exponents satisfying
\[
0<p_-\leq p_+<\infty.
\]

For $p(\cdot)\in\Pcal(\Omega)$, the modular is defined by
\[
\rho_{p(\cdot)}(f)
:=
\int_\Omega
|f(\omega)|^{p(\omega)}
\,d\mathbb{P}(\omega).
\]

The variable Lebesgue space
\[
L_{p(\cdot)}
=
L_{p(\cdot)}(\Omega)
\]
is the set of measurable functions $f$ such that
\[
\rho_{p(\cdot)}(f/\lambda)<\infty
\]
for some $\lambda>0$. It is equipped with the Luxemburg
quasi-norm
\begin{equation}\label{eq:luxemburg}
\norm{f}_{p(\cdot)}
:=
\inf\left\{
\lambda>0:
\rho_{p(\cdot)}(f/\lambda)\leq1
\right\}.
\end{equation}

Put
\[
\underline p
:=
\min\{p_-,1\}.
\]

Then, for all $f,g\in L_{p(\cdot)}$,
\begin{equation}\label{eq:quasi-triangle}
\norm{f+g}_{p(\cdot)}^{\underline p}
\leq
\norm{f}_{p(\cdot)}^{\underline p}
+
\norm{g}_{p(\cdot)}^{\underline p}.
\end{equation}

Consequently, $L_{p(\cdot)}$ is a quasi-Banach space; see
\cite{CruzUribeFiorenza,DieningBook}.

\begin{definition}\label{def:lorentz}
Let $p(\cdot)\in\Pcal(\Omega)$ and $0<q\leq\infty$.

The variable Lorentz space $L_{p(\cdot),q}(\Omega)$ consists
of all measurable functions $f$ for which
\begin{equation}\label{eq:lorentz-norm}
\norm{f}_{L_{p(\cdot),q}}
:=
\begin{cases}
\displaystyle
\left(
\int_0^\infty
\lambda^q
\norm{\1_{\{|f|>\lambda\}}}_{p(\cdot)}^q
\frac{d\lambda}{\lambda}
\right)^{1/q},
&0<q<\infty,
\\[1.2em]
\displaystyle
\sup_{\lambda>0}
\lambda
\norm{\1_{\{|f|>\lambda\}}}_{p(\cdot)},
&q=\infty,
\end{cases}
\end{equation}
is finite.
\end{definition}

We shall use the following elementary scaling property. If
$a>0$ and
\[
r(\cdot)
=
\frac{p(\cdot)}{a},
\]
then, for every measurable set $A\subseteq\Omega$,
\begin{equation}\label{eq:indicator-scaling}
\norm{\1_A}_{r(\cdot)}
=
\norm{\1_A}_{p(\cdot)}^a.
\end{equation}

Moreover, for $0<q<\infty$,
\begin{equation}\label{eq:dyadic-lorentz}
\norm{f}_{L_{p(\cdot),q}}^q
\approxsim
\sum_{k\in\Z}
2^{kq}
\norm{\1_{\{|f|>2^k\}}}_{p(\cdot)}^q,
\end{equation}
and for $q=\infty$,
\begin{equation}\label{eq:dyadic-lorentz-infty}
\norm{f}_{L_{p(\cdot),\infty}}
\approxsim
\sup_{k\in\Z}
2^k
\norm{\1_{\{|f|>2^k\}}}_{p(\cdot)}.
\end{equation}

\subsection{Adapted sequence spaces}
\label{subsec:adapted-spaces}

Let $(\Fcal_n)_{n\geq1}$ be an increasing sequence of
sub-$\sigma$-algebras of $\Fcal$.

A sequence
\[
f=(f_n)_{n\geq1}
\]
is called adapted if $f_n$ is $\Fcal_n$-measurable for every
$n\geq1$.

For an adapted sequence $f$, define
\begin{equation}\label{eq:sigma-def}
\sigma_n(f)
:=
\left(
\sum_{k=1}^{n}|f_k|^2
\right)^{1/2},
\qquad
\sigma(f)
:=
\left(
\sum_{k=1}^{\infty}|f_k|^2
\right)^{1/2}.
\end{equation}

\begin{definition}\label{def:adapted-spaces}
Let $p(\cdot)\in\Pcal(\Omega)$ and $0<q\leq\infty$.

The adapted variable Lebesgue sequence space is
\[
L^{\lad}_{p(\cdot)}
:=
\left\{
f=(f_n)_{n\geq1}:
f\text{ is adapted and }
\sigma(f)\in L_{p(\cdot)}
\right\},
\]
with quasi-norm
\[
\norm{f}_{L^{\lad}_{p(\cdot)}}
:=
\norm{\sigma(f)}_{p(\cdot)}.
\]

The adapted variable Lorentz sequence space is
\[
L^{\lad}_{p(\cdot),q}
:=
\left\{
f=(f_n)_{n\geq1}:
f\text{ is adapted and }
\sigma(f)\in L_{p(\cdot),q}
\right\},
\]
with quasi-norm
\[
\norm{f}_{L^{\lad}_{p(\cdot),q}}
:=
\norm{\sigma(f)}_{L_{p(\cdot),q}}.
\]

Finally,
\[
\norm{f}_{L^{\lad}_\infty}
:=
\norm{\sigma(f)}_\infty.
\]
\end{definition}

\begin{proposition}\label{prop:complete}
For every $p(\cdot)\in\Pcal(\Omega)$, the space
$L^{\lad}_{p(\cdot)}$ is a quasi-Banach space.

More precisely, it is an isometric closed subspace of the
vector-valued space
\[
L_{p(\cdot)}(\Omega;\ell_2).
\]
\end{proposition}

\begin{proof}
The ambient vector-valued space
$L_{p(\cdot)}(\Omega;\ell_2)$ consists of measurable
$\ell_2$-valued functions $g=(g_n)_{n\geq1}$ for which
\[
\norm{g}_{L_{p(\cdot)}(\Omega;\ell_2)}
:=
\norm{
\left(
\sum_{n=1}^{\infty}|g_n|^2
\right)^{1/2}
}_{p(\cdot)}
<\infty.
\]

If $g\in L^{\lad}_{p(\cdot)}$, then
\[
\norm{g}_{L^{\lad}_{p(\cdot)}}
=
\norm{\sigma(g)}_{p(\cdot)}
=
\norm{
\left(
\sum_{n=1}^{\infty}|g_n|^2
\right)^{1/2}
}_{p(\cdot)}
=
\norm{g}_{L_{p(\cdot)}(\Omega;\ell_2)}.
\]
Hence, $L^{\lad}_{p(\cdot)}$ is an isometric linear subspace of
$L_{p(\cdot)}(\Omega;\ell_2)$.

It remains to show that this subspace is closed. Let
\[
f^{(m)}
=
\bigl(f_j^{(m)}\bigr)_{j\geq1}
\in L^{\lad}_{p(\cdot)}
\]
and suppose that
\[
f^{(m)}\longrightarrow f=(f_j)_{j\geq1}
\quad\text{in}\quad
L_{p(\cdot)}(\Omega;\ell_2).
\]
For every fixed $j\geq1$,
\[
|f_j^{(m)}-f_j|
\leq
\left(
\sum_{k=1}^{\infty}
|f_k^{(m)}-f_k|^2
\right)^{1/2},
\]
and therefore
\[
\norm{f_j^{(m)}-f_j}_{p(\cdot)}
\leq
\norm{f^{(m)}-f}_{L_{p(\cdot)}(\Omega;\ell_2)}
\longrightarrow0.
\]
Thus, for each $j$, there exists a subsequence converging to
$f_j$ almost everywhere. Since every $f_j^{(m)}$ is
$\Fcal_j$-measurable and the probability space is complete,
$f_j$ is $\Fcal_j$-measurable. Hence $f$ is adapted, so
$L^{\lad}_{p(\cdot)}$ is closed in
$L_{p(\cdot)}(\Omega;\ell_2)$.

Since $L_{p(\cdot)}(\Omega;\ell_2)$ is complete, it follows
that $L^{\lad}_{p(\cdot)}$ is complete.
\end{proof}

\subsection{The real interpolation method}
\label{subsec:real-interpolation}

Let $(X_0,X_1)$ be a compatible couple of quasi-Banach spaces.
For $x\in X_0+X_1$ and $t>0$, the $K$-functional is defined by
\begin{equation}\label{eq:k-functional}
K(x,t;X_0,X_1)
:=
\inf\left\{
\norm{x_0}_{X_0}
+
t\norm{x_1}_{X_1}:
x=x_0+x_1
\right\}.
\end{equation}

Let $0<\theta<1$ and $0<q\leq\infty$. The real interpolation
space $(X_0,X_1)_{\theta,q}$ consists of all
$x\in X_0+X_1$ such that
\begin{equation}\label{eq:interpolation-norm}
\norm{x}_{(X_0,X_1)_{\theta,q}}
:=
\begin{cases}
\displaystyle
\left(
\int_0^\infty
\bigl[
t^{-\theta}
K(x,t;X_0,X_1)
\bigr]^q
\frac{dt}{t}
\right)^{1/q},
&0<q<\infty,
\\[1.2em]
\displaystyle
\sup_{t>0}
t^{-\theta}
K(x,t;X_0,X_1),
&q=\infty,
\end{cases}
\end{equation}
is finite.

The following scalar interpolation identity is due to Kempka
and Vyb\'iral \cite{KempkaVybiral}.

\begin{lemma}\label{lem:scalar-interpolation}
Let $p(\cdot)\in\Pcal(\Omega)$,
$0<q\leq\infty$, and $0<\theta<1$.

Define $\widetilde p(\cdot)$ by
\[
\frac{1}{\widetilde p(\cdot)}
=
\frac{1-\theta}{p(\cdot)}.
\]
Then
\begin{equation}\label{eq:scalar-interpolation}
\bigl(
L_{p(\cdot)},
L_\infty
\bigr)_{\theta,q}
=
L_{\widetilde p(\cdot),q}
\end{equation}
with equivalent quasi-norms.
\end{lemma}

\section{Proof of
\texorpdfstring{\Cref{thm:main}}{Theorem 1.1}}
\label{sec:proof-main}

We first prove the adapted truncation estimate required for the
reverse embedding.

\begin{lemma}\label{lem:k-estimate}
Let $p(\cdot)\in\Pcal(\Omega)$. For every
\[
f\in
L^{\lad}_{p(\cdot)}
+
L^{\lad}_{\infty}
\]
and every $t>0$,
\begin{equation}\label{eq:k-estimate}
K\bigl(
f,t;
L^{\lad}_{p(\cdot)},
L^{\lad}_{\infty}
\bigr)
\leq
\inf_{\mu>0}
\left\{
\norm{
\sigma(f)
\1_{\{\sigma(f)>\mu\}}
}_{p(\cdot)}
+
t\mu
\right\}.
\end{equation}
\end{lemma}

\begin{proof}
Fix $\mu>0$. For $k\geq1$, set
\[
Q_k
=
\{\sigma_k(f)\leq\mu\}.
\]
Then $Q_k\in\Fcal_k$, the sequence $(Q_k)_{k\geq1}$ is
decreasing, and
\[
Q
:=
\bigcap_{k\geq1}Q_k
=
\{\sigma(f)\leq\mu\}.
\]

Define
\[
g_k
:=
f_k\1_{Q_k},
\qquad
h_k
:=
f_k\1_{\Omega\setminus Q_k}.
\]
Then $g=(g_k)$ and $h=(h_k)$ are adapted and $f=g+h$.

We first show that
\[
\norm{g}_{L^{\lad}_\infty}
\leq\mu.
\]
For every $n\geq1$,
\begin{align*}
\sigma_n^2(g)
&=
\sum_{k=1}^{n}|g_k|^2
\\
&=
\sum_{k=1}^{n}|f_k|^2\1_{Q_k}
\\
&=
\sum_{k=1}^{n}
\bigl(
\sigma_k^2(f)-\sigma_{k-1}^2(f)
\bigr)
\1_{Q_k}
\\
&=
\sum_{k=1}^{n}
\sigma_k^2(f)\1_{Q_k}
-
\sum_{k=1}^{n}
\sigma_{k-1}^2(f)\1_{Q_k}.
\end{align*}
After shifting the index in the second sum, we obtain
\[
\sigma_n^2(g)
=
\sigma_n^2(f)\1_{Q_n}
+
\sum_{k=1}^{n-1}
\sigma_k^2(f)
\bigl(
\1_{Q_k}-\1_{Q_{k+1}}
\bigr).
\]
By the definition of $Q_k$,
\[
\sigma_n^2(g)
\leq
\mu^2\1_{Q_n}
+
\mu^2
\sum_{k=1}^{n-1}
\bigl(
\1_{Q_k}-\1_{Q_{k+1}}
\bigr)
\leq
\mu^2.
\]
Therefore,
\[
\sigma(g)\leq\mu
\]
and hence
\[
\norm{g}_{L^{\lad}_\infty}
\leq\mu.
\]

For $h$, we have
\begin{align*}
\sigma^2(h)
&=
\sum_{k\geq1}
|f_k|^2
\1_{\Omega\setminus Q_k}
\\
&\leq
\sum_{k\geq1}
|f_k|^2
\1_{\Omega\setminus Q}
\\
&=
\sigma^2(f)
\1_{\Omega\setminus Q}.
\end{align*}
Consequently,
\[
\norm{h}_{L^{\lad}_{p(\cdot)}}
\leq
\norm{
\sigma(f)
\1_{\{\sigma(f)>\mu\}}
}_{p(\cdot)}.
\]

Thus,
\begin{align*}
K\bigl(
f,t;
L^{\lad}_{p(\cdot)},
L^{\lad}_{\infty}
\bigr)
&\leq
\norm{h}_{L^{\lad}_{p(\cdot)}}
+
t\norm{g}_{L^{\lad}_{\infty}}
\\
&\leq
\norm{
\sigma(f)
\1_{\{\sigma(f)>\mu\}}
}_{p(\cdot)}
+
t\mu.
\end{align*}
Taking the infimum over $\mu>0$ proves the lemma.
\end{proof}

\begin{proof}[Proof of \Cref{thm:main}]
We first prove
\[
\bigl(
L^{\lad}_{p(\cdot)},
L^{\lad}_{\infty}
\bigr)_{\theta,q}
\subset
L^{\lad}_{\widetilde p(\cdot),q}.
\]

Define the sublinear operator
\[
T(f)
:=
\sigma(f).
\]
By definition,
\[
\norm{T(f)}_{p(\cdot)}
=
\norm{\sigma(f)}_{p(\cdot)}
=
\norm{f}_{L^{\lad}_{p(\cdot)}}
\]
and
\[
\norm{T(f)}_{\infty}
=
\norm{\sigma(f)}_{\infty}
=
\norm{f}_{L^{\lad}_{\infty}}.
\]
Moreover, by the Minkowski inequality in $\ell_2$,
\[
T(f+g)
=
\sigma(f+g)
\leq
\sigma(f)+\sigma(g)
=
T(f)+T(g).
\]
Therefore, by the interpolation theorem for sublinear operators,
see Theorem 5.2.3 of Bergh and L\"ofstr\"om
\cite{BerghLofstrom},
\[
T:
\bigl(
L^{\lad}_{p(\cdot)},
L^{\lad}_{\infty}
\bigr)_{\theta,q}
\longrightarrow
\bigl(
L_{p(\cdot)},
L_\infty
\bigr)_{\theta,q}
\]
is bounded. Applying \Cref{lem:scalar-interpolation}, we obtain
\[
\bigl(
L_{p(\cdot)},
L_\infty
\bigr)_{\theta,q}
=
L_{\widetilde p(\cdot),q},
\]
and hence
\[
\norm{f}_{L^{\lad}_{\widetilde p(\cdot),q}}
=
\norm{\sigma(f)}_{L_{\widetilde p(\cdot),q}}
\lesssim
\norm{f}_{
\left(
L^{\lad}_{p(\cdot)},
L^{\lad}_{\infty}
\right)_{\theta,q}
}.
\]
This proves the first embedding.

We now prove the converse embedding
\[
L^{\lad}_{\widetilde p(\cdot),q}
\subset
\bigl(
L^{\lad}_{p(\cdot)},
L^{\lad}_{\infty}
\bigr)_{\theta,q}.
\]

Assume first that $0<q<\infty$. Put
\[
\underline p
=
\min\{1,p_-\}
\]
and define
\[
h(\lambda)
:=
\norm{
\1_{\{\sigma(f)>\lambda\}}
}_{p(\cdot)}.
\]

For every $\mu>0$,
\[
\sigma(f)\1_{\{\sigma(f)>\mu\}}
\lesssim
\sum_{j=0}^{\infty}
2^j\mu\,
\1_{\{\sigma(f)>2^j\mu\}}.
\]
Indeed,
\[
(\mu,\infty)
=
\bigcup_{j=0}^{\infty}
(2^j\mu,2^{j+1}\mu],
\]
and therefore the preceding estimate follows by considering the
dyadic interval containing $\sigma(f)$.

Using \Cref{lem:k-estimate} and the quasi-triangle inequality
\Cref{eq:quasi-triangle}, we obtain
\begin{align*}
K\bigl(
f,t;
L^{\lad}_{p(\cdot)},
L^{\lad}_{\infty}
\bigr)
&\lesssim
\inf_{\mu>0}
\left\{
\norm{
\sigma(f)
\1_{\{\sigma(f)>\mu\}}
}_{p(\cdot)}
+
t\mu
\right\}
\\
&\lesssim
\inf_{\mu>0}
\left\{
\left(
\sum_{j=0}^{\infty}
2^{j\underline p}
h(2^j\mu)^{\underline p}
\right)^{1/\underline p}
\mu
+
t\mu
\right\}.
\end{align*}

For $t>0$, define
\[
\mu(t)
:=
\inf
\left\{
\mu>0:
\left(
\sum_{j=0}^{\infty}
2^{j\underline p}
h(2^j\mu)^{\underline p}
\right)^{1/\underline p}
\leq t
\right\}.
\]
By approximation from above and the monotonicity of $h$, one has
\[
\left(
\sum_{j=0}^{\infty}
2^{j\underline p}
h(2^j\mu(t))^{\underline p}
\right)^{1/\underline p}
\leq t.
\]
Consequently,
\[
K\bigl(
f,t;
L^{\lad}_{p(\cdot)},
L^{\lad}_{\infty}
\bigr)
\lesssim
t\mu(t).
\]

Therefore,
\begin{align*}
\int_0^\infty
t^{-\theta q}
K(f,t)^q
\frac{dt}{t}
&\lesssim
\int_0^\infty
t^{(1-\theta)q}
\mu(t)^q
\frac{dt}{t}
\\
&=
\sum_{k\in\Z}
\int_{\{t:2^k<\mu(t)\leq2^{k+1}\}}
t^{(1-\theta)q}
\mu(t)^q
\frac{dt}{t}
\\
&\lesssim
\sum_{k\in\Z}
2^{kq}
\int_{\{t:\mu(t)>2^k\}}
t^{(1-\theta)q}
\frac{dt}{t}.
\end{align*}

If $\mu(t)>2^k$, then
\[
t
\leq
\left(
\sum_{j=0}^{\infty}
2^{j\underline p}
h(2^{j+k})^{\underline p}
\right)^{1/\underline p}.
\]
Hence,
\begin{align*}
\int_0^\infty
t^{-\theta q}
K(f,t)^q
\frac{dt}{t}
&\lesssim
\sum_{k\in\Z}
2^{kq}
\left[
\sum_{j=0}^{\infty}
2^{j\underline p}
h(2^{j+k})^{\underline p}
\right]^{(1-\theta)q/\underline p}.
\end{align*}

Set
\[
s
:=
\frac{(1-\theta)q}{\underline p}.
\]

If $s\leq1$, then
\[
\left(
\sum_{j=0}^{\infty}a_j
\right)^s
\leq
\sum_{j=0}^{\infty}a_j^s
\]
for nonnegative $(a_j)$. Therefore,
\begin{align*}
\int_0^\infty
t^{-\theta q}
K(f,t)^q
\frac{dt}{t}
&\lesssim
\sum_{k\in\Z}
2^{kq}
\sum_{j=0}^{\infty}
2^{j(1-\theta)q}
h(2^{j+k})^{(1-\theta)q}
\\
&=
\sum_{\ell\in\Z}
2^{\ell q}
h(2^\ell)^{(1-\theta)q}
\sum_{j=0}^{\infty}
2^{-\theta jq}
\\
&\lesssim
\sum_{\ell\in\Z}
2^{\ell q}
h(2^\ell)^{(1-\theta)q}.
\end{align*}

If $s>1$, choose
\[
0<\varepsilon
<
\frac{\theta}{1-\theta}.
\]
By H\"older's inequality,
\[
\left(
\sum_{j=0}^{\infty}
2^{j\underline p}
h(2^{j+k})^{\underline p}
\right)^s
\lesssim
\sum_{j=0}^{\infty}
2^{(\varepsilon+1)(1-\theta)jq}
h(2^{j+k})^{(1-\theta)q}.
\]
Consequently,
\begin{align*}
\int_0^\infty
t^{-\theta q}
K(f,t)^q
\frac{dt}{t}
&\lesssim
\sum_{k\in\Z}
2^{kq}
\sum_{j=0}^{\infty}
2^{(\varepsilon+1)(1-\theta)jq}
h(2^{j+k})^{(1-\theta)q}
\\
&=
\sum_{\ell\in\Z}
2^{\ell q}
h(2^\ell)^{(1-\theta)q}
\sum_{j=0}^{\infty}
2^{(\varepsilon(1-\theta)-\theta)jq}
\\
&\lesssim
\sum_{\ell\in\Z}
2^{\ell q}
h(2^\ell)^{(1-\theta)q}.
\end{align*}

By the scaling identity \Cref{eq:indicator-scaling},
\[
\norm{
\1_{\{\sigma(f)>2^\ell\}}
}_{\widetilde p(\cdot)}
=
\norm{
\1_{\{\sigma(f)>2^\ell\}}
}_{p(\cdot)}^{1-\theta}.
\]
Therefore, using \Cref{eq:dyadic-lorentz},
\begin{align*}
\int_0^\infty
t^{-\theta q}
K(f,t)^q
\frac{dt}{t}
&\lesssim
\sum_{\ell\in\Z}
2^{\ell q}
\norm{
\1_{\{\sigma(f)>2^\ell\}}
}_{\widetilde p(\cdot)}^q
\\
&\approxsim
\norm{
\sigma(f)
}_{L_{\widetilde p(\cdot),q}}^q.
\end{align*}
Thus,
\[
\norm{f}_{
\left(
L^{\lad}_{p(\cdot)},
L^{\lad}_{\infty}
\right)_{\theta,q}
}
\lesssim
\norm{f}_{L^{\lad}_{\widetilde p(\cdot),q}}.
\]

For $q=\infty$, the same argument is carried out.
Using \Cref{eq:dyadic-lorentz-infty}, we obtain
\[
\sup_{t>0}
t^{-\theta}K(f,t)
\lesssimc
\sup_{\ell\in\Z}
2^\ell
h(2^\ell)^{1-\theta}
\approxsim
\norm{
\sigma(f)
}_{L_{\widetilde p(\cdot),\infty}}.
\]

Hence, the second embedding holds for all
$0<q\leq\infty$. Combining the two embeddings completes the
proof.
\end{proof}


\section*{Funding}
This research did not receive any specific grant from funding agencies
in the public, commercial, or not-for-profit sectors.

\section*{Declaration of competing interest}
The author declares that he has no known competing financial interests
or personal relationships that could have appeared to influence the
work reported in this paper.

\section*{Data availability}
No data were used for the research described in this article.

\section*{Declaration of generative AI and AI-assisted technologies in the manuscript preparation process}
During the preparation of this work, the author used ChatGPT (OpenAI)
to assist with language editing, manuscript organization, and LaTeX
formatting. After using this tool, the author reviewed and edited the
content as needed and takes full responsibility for the content of the
published article.

\end{document}